\documentclass[10pt]{amsart}

\usepackage[english]{babel}

\usepackage[all]{xy}
\usepackage{amssymb,amsmath}
\usepackage[mathscr]{euscript}
\usepackage{mathrsfs}

\newtheorem{thm}{Theorem}[section]
\newtheorem{prop}[thm]{Proposition}
\newtheorem{cor}[thm]{Corollary}
\newtheorem{lm}[thm]{Lemma}

\newtheorem*{thm*}{Theorem}
\newtheorem*{prop*}{Proposition}

\theoremstyle{definition}
\newtheorem{df}[thm]{Definition}
\newtheorem{ex}[thm]{Example}

\newtheorem{num}[thm]{}

\theoremstyle{remark}
\newtheorem{rem}[thm]{Remark}

\numberwithin{equation}{thm}

\newcommand{\dmgme} {DM_{gm}^{eff}\!(k)}
\newcommand{\dmgm} {DM_{gm}\!(k)}
\newcommand{\dmgmo} {DM_{gm}^{(0)}\!(k)}

\newcommand{\sm}{\mathscr Sm_k}
\newcommand{\smc}{\mathscr Sm^{cor}_k}

\newcommand{\T}{\mathscr T}
\newcommand{\A}{\mathscr A}

\newcommand{\cH}{\mathcal H}

\newcommand{\drap}[1] {\mathcal D \left( #1 \right)}

\newcommand{\Hom}{\operatorname{\mathrm{Hom}}}

\newcommand{\mot}[1] {M\!\left( #1 \right)}

\newcommand{\M}{\mathcal M}

\newcommand{\spec}[1] {\mathrm{Spec}(#1)}

\newcommand{\ilim}[1] { \underset{#1}{\varinjlim} \ }

\newcommand{\pplim}[1] { \underset{#1}{"{\varprojlim}"} \ }
\newcommand{\pprod}[1] { \underset{#1}{"{\prod}"} \ }

\newcommand{\tra}{{}^t}

\newcommand{\K}{\mathcal K}

\newcommand{\pur}[1] { \mathfrak p_{(#1)} }

\newcommand{\dte} { \mathbb A^1_k }

\newcommand{\dtex}[1] { \mathbb A^1_{#1} }
\newcommand{\dtenx}[2] { \mathbb A^{#1}_{#2} }
\newcommand{\GG} { \mathbb G_m }
\newcommand{\PP} { \mathbb P }

\newcommand{\NN} {\mathbb N}
\newcommand{\ZZ} {\mathbb Z}

\newcommand{\nis}{\mathrm{Nis}}

\newcommand{\E}{\mathbb E}

\newcommand{\dtwist}[1]{\!((#1))}

\newcommand{\pro}[1]{\mathrm{pro}\!-\!{#1}}

\begin{document}

\title{Coniveau filtration and mixed motives}

\author{Fr\'ed\'eric D\'eglise}
\address{LAGA\\
CNRS~(UMR 7539)\\
Universit\'e Paris~13\\
\hbox{Avenue~Jean-Baptiste~Cl\'ement}\\
93430 Villetaneuse\\France}
\email{deglise@math.univ-paris13.fr}
\urladdr{http://www.math.univ-paris13.fr/~deglise/}
\thanks{Partially supported by the ANR (grant No. ANR-07-BLAN-042)}

\date{June 2011}

\begin{abstract}
We introduce the motivic coniveau exact couple of a scheme,
 in the framework of mixed motives, whose property is to universally give
 rise to coniveau spectral sequences through realizations.
 The main result is a computation of its differentials
 in terms of residues and transfers of mixed motives,
 with a formula analog to the one defining the Weil divisor of a
 rational function. 
 We then show how to recover and extend classical results of 
 Bloch and Ogus for motivic realizations.
\end{abstract}

\maketitle

\section*{Introduction}

The coniveau filtration is deeply rooted in the main conjectures on algebraic cycles,
 the ones of Hodge and Tate, as stated by Grothendieck in \cite{Gro1} and \cite{Gro2}.
 This filtration was first define on De Rham cohomology by Grothendieck as the abutment of a spectral sequence deduced from the Cousin resolution of coherent sheaves
 (see \cite[footnote (8), p. 356-357]{Gro4}). In fact, as it appears in
 \cite[chap. IV]{Har}, Cousin resolutions can be understood as the lines in the $E_1$-term
 of a suitable sheafified spectral sequence for cohomology with coefficients in a coherent sheaf. This kind of spectral sequence is now called \emph{coniveau spectral sequence} after the landmarking work of Bloch and Ogus \cite{BO}.

As higher algebraic K-theory was rising, Gersten extended the preceding considerations to the case of K-theory ending in the celebrated \emph{Gersten conjecture}.\footnote{Recall 
 this conjecture can be phrased by saying that 
 the Cousin complex of the unramified K-theory sheaf over a regular scheme 
 is a resolution: see \cite[Problem 10, p. 46]{Ger}. Following \cite[Def. p. 238]{Har},
 one also says this sheaf is \emph{Cohen-Macaulay}.}
At that time, Bloch discovered the connection of the coniveau spectral sequence in K-theory
 with algebraic cycles, formulating what is now called, after Quillen,
  \emph{Bloch's formula}.
 In fact this formula is a consequence of the Gersten conjecture,
  as showed by Quillen as a corollary of his proof of the conjecture
  in the equicaracteristic case (\cite[5.19]{quillen}). It is important for us to recall 
  that Quillen deduces this fact from the identification of some of the differentials 
  in the $E_1$-term of the coniveau spectral sequence with the classical divisor class map,
  associating to a rational function its Weil divisor (\cite[proof of 5.14]{quillen}).
The purpose of this article is to extend this computation in the theory of Voevodsky's motives.

Motivated by the circle of ideas around the Gersten conjecture,
 Rost introduced in \cite{Ros} a notion of local system, called \emph{cycle modules},
 which gives a theoretical framework to define a generalized divisor class map
 for certain cycles with coefficients. The primary example of a cycle module
 is the well-known Milnor K-functor $K_*^M$.
 As an illustration,
  we recall the definition of the generalized divisor class map for
  a normal algebraic connected $k$-scheme $X$, $k$ being a field.
 In this case, any codimension $1$ point $x$ of $X$ corresponds 
  to a discrete valuation $v_x$ of the function field $E$ of $X$,
  whose residue field $\kappa(x)$ is the residue field of $x$ in $X$.
 To the valued field $(E,v_x)$ is associated the so called \emph{tame symbol map}
 (cf \cite[2.1]{Milnor}):
$$
\partial_{v_x}:K_*^M(E) \rightarrow K_*^M(\kappa(x)),
$$
an homogeneous morphism of graded abelian groups of degree $-1$
 which in degree $1$ is equal to the valuation $v_x$ itself.
 We define the divisor class map as the following morphism:
$$
d^0_X:K_*^M(E) \xrightarrow{\ \sum_x \partial_{v_x}\ }
 \bigoplus_{x \in X^{(1)}} K_*^M(\kappa(x)).
$$
This is an homogeneous morphism of degree $-1$.
According to what was said before,
 the reader can see that, in degree $1$,
 it is precisely the usual divisor class map.
  
One can extend this formula to arbitrary algebraic $k$-schemes
 by using normalization of schemes and transfers in Milnor K-theory.\footnote{The
 divisor class map in Milnor K-theory was first written down by
 Kato in \cite{Kato}, for arbitrary excellent schemes.} 
 The theory of cycle modules, 
 or rather the intermediary notion of a \emph{cycle premodule}, 
 is an axiomatization of the functorial properties of Milnor K-theory,
 as a functor defined over function fields over $k$,
 which allows to use the same formula.
 In the end, one obtains for a cycle premodule $M$ over $k$,
  an algebraic $k$-scheme $X$
  and an integer $n \geq 0$ a canonical morphism
$$
d^n_{X,M}:\bigoplus_{x \in X^{(n)}} M(\kappa(x))
 \longrightarrow \prod_{y \in X^{(n+1)}} M(\kappa(y))
$$
homogeneous of degree $-1$. 
 Recall the first axiom of a cycle module says 
 this map actually lands in the direct sum over $y \in X^{(n+1)}$.

The bridge between Rost's cycle modules and Voevodsky's motives was built
 in the thesis of the author. We defined two reciprocal equivalences of
 categories between cycle modules and certain graded sheaves obtained 
 by a stabilization process from the homotopy invariant sheaves with transfers 
 of the theory of Voevodsky (see \cite[Th. 3.4]{Deg9}). 
 This can be seen as an elaboration
 on the fact these latter sheaves admit a Gersten resolution\footnote{In
 other words, their restriction to any smooth $k$-scheme are Cohen-Macaulay
 in the sense of \cite[Def. p. 238]{Har}.}
 as proved by Voevodsky (\cite[chap. 3, 4.37]{FSV}).
Of the results of our thesis, we will need only the following one:
\begin{thm*}[\cite{Deg5bis}, th. 6.2.1]
Let $k$ be a perfect field,
 $\dmgm$ be the category of geometrical non effective motives of Voevodsky over $k$
 and $\A$ be a Grothendieck abelian category.

Let $H:\dmgm^{op} \rightarrow \A$ be a cohomological functor\footnote{i.e. sending
 distinguished triangles to long exact sequences.}.
 Then for any couple $(q,n) \in \ZZ^2$, there exists a canonical cycle premodule 
 $\hat H^{q,n}_*$ with coefficients\footnote{Actually,
  Rost in \cite{Ros} defines cycles modules
 and premodules only with coefficients in the category of abelian groups but the
 generalization to arbitrary abelian category is immediate.}
 in $\A$
 such that for any integral $k$-scheme $X$ with function field $E$,
$$
\hat H^{q,n}_r(E)=\ilim{U \subset X} H\big(M(U)(n+r)[q+r]\big)
$$
where the limit runs over non empty smooth open subschemes $U$ of $X$,
 $M(U)$ being the motive associated with $U$.
\end{thm*}
Given a functor $H$ as in the above theorem,
 one can apply the usual considerations alluded to above and
 get a coniveau spectral sequence. To describe it,
 we introduce the following useful notation:
 for any smooth $k$-scheme $X$ and any triple of integers $(q,n,r)$,
 we put:
$$
H^{q,n}_r(X):=H\big(M(X)(q+r,n+r)\big).
$$
In our case, given a smooth $k$-scheme $X$ and an integer $n \in \ZZ$,
 the coniveau spectral sequence takes the form:
\begin{equation} \tag{$*$} \label{eq:intro:coniv_spectral}
E_1^{p,q}(X,n)=\bigoplus_{x \in X^{(p)}} \hat H^{q,n}_*(\kappa(x))
 \Rightarrow H^{p+q,n}_*(X),
\end{equation}
with coefficients in the category of graded objects of $\A$,
 differentials in the $E_1$-page being homogeneous of degree $-1$.
 It converges as required to the coniveau filtration
  on the cohomology $H$
 -- see \eqref{coniv_spectral} in the text.
 Explicitly, the $p$-th filtered part of this filtration is made of 
  cohomology classes of $X$ 
  which vanish on an open subscheme of $X$
  with complement of codimension at least $p$ in $X$.

With these notations, we can now state the main result of this paper:
\begin{prop*}[Prop. \ref{realisation:coniveau&Rost}]
Consider a cohomological functor $H:\dmgm^{op} \rightarrow \A$
 as above and denote by $d_1^{p,q}$ the differentials of
 the spectral sequence \eqref{eq:intro:coniv_spectral}.
 
Then for any couple of integers $(p,q) \in \ZZ^2$,
$$
d_1^{p,q}=d^p_{\hat H^{q,n}}.
$$
\end{prop*}
This gives back the computation of Quillen but replacing K-theory with motivic cohomology
 -- it can be shown this amounts to the same using the more involved work
 of \cite{Deg10}. The most interesting fact is that we obtain as a corollary
 that the divisor maps $d^*_{\hat H^{q,n}_*}$ induce a well defined complex: 
 in the terminology of Rost, $\hat H^{q,n}_*$ is a cycle module.
 We already proved this fact in \cite{Deg5bis} using results of Rost and the fact
 the field $k$ is perfect. The proof given here is much more direct and geometric.
 To be precise, it relies on our work on the Gysin triangle \cite{Deg6} --
 the main ingredient is the functoriality of residues with respect to Gysin
 morphisms.

As apparent from the beginning,
 this study is closely connected with the pioneering work of Bloch and Ogus.
 In fact, as a corollary of the preceding computation together with
 the results of \cite{Ros} and \cite{Deg9}, we get back the results of
 Bloch and Ogus for the cohomology $H$. More precisely:
\begin{thm*}[Prop. \ref{prop:coniv_spectral3} and \ref{prop:comput_unramified_sheaf}]
Consider as above a cohomological functor $H:\dmgm^{op} \rightarrow \A$.
For any smooth scheme $X$,
 let us denote by $\mathcal H^{q,n}_*(X)$ the kernel of the divisor map
 $d^0_{X,\hat H^{q,n}_*}$.

Then $\mathcal H^{q,n}_*$ is a homotopy invariant Nisnevich sheaf
 with transfers in the sense of Voevodsky. It coincides with
 the Zariski sheaf associated with $H^{q,n}_*$.

The coniveau spectral sequence \eqref{eq:intro:coniv_spectral}
 can be written from $E_2$ on as:
$$
E_2^{p,q}(X,n)=H^p_\mathrm{Zar}(X,\mathcal H^{q,n}_*)
 \simeq H^p_\nis(X,\mathcal H^{q,n}_*)
 \Rightarrow H^{p+q,n}_*(X).
$$
\end{thm*}

Let us finish this introduction with a concrete application of
 these theoretical results.
Let $k$ be a perfect field of characteristic $p>0$,
 $W$ its Witt ring and $K$ the fraction field of $W$.
 Given a smooth (resp. smooth affine) $k$-scheme $X$,
  we denote by $H^*_{crys}(X/W)$ 
  (resp. $H^*_{MW}(X)$)
 the crystalline cohomology (resp. Monsky-Washnitzer) of $X$ defined in \cite{Ber0}
  (resp. \cite{MW}).

Then the following properties hold:
\begin{itemize}
\item Let $\cH^*_{MW}$ be the Zariski sheaf on the category of smooth $k$-schemes
 associated with the presheaf $H^*_{MW}$.

Then $\cH^*_{MW}$ is a homotopy invariant Nisnevich sheaf with transfers
 and $\Gamma(X,\cH^*_{MW})$ is invariant
  on the birational class of a smooth proper scheme $X$.
\item For any smooth proper scheme $X$,
 there exists a spectral sequence
$$
E_2^{p,q}=H^p_{\mathrm{Zar}}(X,\cH^q_{MW})
 \Rightarrow H^{p+q}_{crys}(X/W) \otimes K
$$
converging to the coniveau filtration on $H^i_{crys}(X/W)_K$.
\item When $k$ is separably closed,
 for any $p \geq 0$,
$$
H^p_{\mathrm{Zar}}(X,\cH^p_{MW})
=A^p(X) \otimes K
$$ where the right hand side is the group
  of $p$-codimensional cycles modulo algebraic equivalence.
\end{itemize}
This set of properties is a corollary of the preceding theorem
 when one uses the \emph{rigid cohomology} defined by Berthelot
 (e.g. \cite{Ber1})
 together with its realization
  $H_{rig}:\dmgm^{op} \rightarrow K\!-\!vs$ introduced in the work of
  Cisinski and the author on Mixed Weil cohomologies \cite{CD2}
  -- see Remark \ref{rigid_realization} for details.

\subsubsection*{Organization of the paper}
Section 1 is the technical heart of the paper. 
 We introduce the new notion of a triangulated exact couple (Def. \ref{triang_e_couple}),
 associated with any filtered object of a triangulated category together
 with a choice of cones.
 Our main example is what we call the \emph{motivic coniveau exact couple}
 (Def. \ref{coniv_exact_couple})
  which, through a realization functor of mixed motives,
  universally gives rise to the exact couple corresponding
  to the coniveau spectral sequence.
 In section \ref{sec:computations} we give the computation of the differentials 
 of this exact couple in terms of \emph{generic motives}
  (recalled in section \ref{sec:generic_motives}).
Section 2 relates this computation with the theory of cycle
modules through cohomological realizations and gives a proof of the results
stated in this introduction.

\subsubsection*{Foreword}
There has been a lot of work on the coniveau spectral sequence
 apart those cited in the introduction.
 Let us mention in particular \cite{CTHK} which gives,
 using a proof of Gabber,
 the computation of the $E_2$-term as in the above theorem
  in a more general setting -- which does not require homotopy invariance.
 It can be applied to the example given above by using in particular
 the \'etale descent theorem of \cite{CT} for rigid cohomology
 -- this is well known to imply Nisnevich excision as required in \cite{CTHK}.

Our contribution to the story is made of the following points:
\begin{itemize}
\item The unramified cohomology sheaves $\cH^{q,n}$, in our setting,
 admits a canonical structure of a homotopy invariant sheaf with transfers
  -- a particular case of \emph{motivic complex}.
\item The differentials in the $E_1$-term can be computed
 in terms of Rost theory through an explicit cycle module.
\item We have extended the identification of the diagonal terms
 in the $E_2$-page of the spectral
 sequence to cycles modulo algebraic equivalence,
  proved in characteristic $0$ in \cite{BO},
 to the case of characteristic $p$ for a reasonable cohomology theory
 (see the example above and Corollary \ref{Bloch-Ogus_algebraic}
 for the general case).
\end{itemize}

\section*{Notations and conventions}

We fix a base field $k$ which is assumed to be perfect. 
The word scheme will stand for any
separated $k$-scheme of finite type, and we will say that a scheme is
smooth when it is smooth over the base field. 
The category of smooth schemes
is denoted by $\sm$.

We will also use the theory of geometric motives defined by Voevodsky
 in \cite[chap. 5]{FSV}).
 Therefore, 
  we denote by $\smc$ the category of smooth schemes with morphisms
  the finite correspondences.
 This is an additive category whose objects are denoted under bracket
 -- eg $\lbrack X \rbrack$.
 The category of geometric effective motives, denoted by $\dmgme$,
  is a quotient of the bounded homotopy category associated with $\smc$.
 The motive $M(X)$ of a smooth scheme $X$ is the complex
  equal to $\lbrack X \rbrack$ placed in degree $0$.
 Given an open subscheme $U$ of a smooth scheme $X$,
  we will define the \emph{relative motive} of $X$ modulo $U$ as the
  cone of the canonical immersion $U \rightarrow X$ computed
  in the category of complexes with coefficients in $\smc$:
$$
M(X/U):=\mathrm{Cone}(\lbrack U \rbrack \rightarrow \lbrack X \rbrack).
$$
This motive is functorial with respect to morphisms of schemes which stabilize
 the open subscheme.

The category $\dmgme$ is monoidal, with unit $\ZZ:=M(\spec k)$.
 Recall the Tate twist can be defined as:
$$
\ZZ(1):=M(\dte/\GG)[-2].
$$
The category of geometric motives $\dmgm$
 is the stabilization of $\dmgme$ with respect to $\ZZ(1)$ --
 \emph{i.e.} the monoidal category obtained by formally inverting this object
 with respect to the tensor product.
 There is a canonical functor
$$
\Sigma^\infty:\dmgme \rightarrow \dmgm
$$
 and we will still denote by $M(X)$ (resp. $M(X/U)$)
 the image of $M(X)$ (resp. $M(X/U)$) by $\Sigma^\infty$.

According to the theory of Voevodsky,
 a \emph{homotopy invariant sheaf with transfers} with
 values in an abelian category $\A$ will be
 a functor $F:(\smc)^{op} \rightarrow \A$ which is a sheaf
 for the Nisnevich topology and such that for any smooth scheme $X$,
 the map $F(X) \rightarrow F(\dtex X)$ induced by the projection
 is an isomorphism.

\bigskip

We will use the following constructions of \cite{Deg6}:
\begin {itemize}
\item Given a pair $(X,Z)$ such that $X$ (resp. $Z$) is a smooth scheme
(resp. smooth closed subscheme of $X$ of pure codimension $n$),
 we defined a \emph{purity isomorphism} in \cite[Prop. 1.12]{Deg6}:
\begin{equation} \tag{Intro.a}\label{eq:purity}
\pur{X,Z}:M(X/X-Z) \rightarrow M(Z)(n)[2n].
\end{equation}
\item Given a projective morphism $f:Y \rightarrow X$ between smooth schemes
 of pure dimension $d$, we defined a \emph{Gysin morphism}
 in \cite[Def. 2.7]{Deg6}: 
\begin{equation} \tag{Intro.b}\label{eq:gysin}
f^*:M(X)(d)[2d] \rightarrow M(Y).
\end{equation}
\end{itemize}

\tableofcontents

\section{Motivic coniveau exact couple}

\subsection{Definition}

\subsubsection{Triangulated exact couple}

We introduce a triangulated version of the classical exact 
couples.
\begin{df} \label{triang_e_couple}
Let $\T$ be a triangulated category.
A triangulated exact couple is the data of bigraded
objects $D$ and $E$ of $\T$ and homogeneous morphisms 
between them
\begin{equation}
\label{hlg_exact_couple}
\xymatrix@=28pt{
D\ar^{(1,-1)}_\alpha[rr]
 &&  D\ar^{(0,0)}_/8pt/\beta[ld] \\
 & E\ar^{(-1,0)}_/-6pt/\gamma[lu] &
}
\end{equation}
with the bidegrees of each morphism indicated in
the diagram and such that
the above triangle is a distinguished triangle in each
bidegree.\footnote{Note this implies in particular the relation
$D_{p,q+1}=D_{p,q}[-1]$ for any couple of integers $(p,q)$.}
\end{df}

Given such a triangulated exact couple,
 we will usually put $d=\beta \circ \gamma$, 
 homogeneous endomorphism of $E$ of bidegree $(-1,0)$. 
We easily get that $d^2=0$, thus obtaining a complex
$$
... \rightarrow E_{p,q} \xrightarrow{d_{p,q}} E_{p-1,q}
 \rightarrow ...
$$

Let $\A$ be an abelian category. A cohomological functor
with values in $\A$ is an additive functor
$H:\T^{op} \rightarrow \A$ which sends distinguished 
triangles to long exact sequences. For $p$ an integer, 
we simply put $H^p=H \circ .[-p]$. \\
Apply the contravariant functor $H=H^0$
 to the diagram \eqref{hlg_exact_couple}, 
 we naturally obtain a commutative diagram of bigraded
 objects of $\A$:
$$
\xymatrix@=28pt{
H(D)\ar_{(1,0)}^/+6pt/{\gamma^*}[rd]
 &&  H(D)\ar_{(-1,1)}^{\alpha^*}[ll]
  \\
 & H(E)\ar_{(0,0)}^/-8pt/{\beta^*}[ru] &
}
$$
This is an exact couple of $\A$ 
in the classical sense (following the convention of \cite[th. 2.8]{McC}). 
Thus we can associate with this exact couple a spectral
sequence:
$$
E_1^{p,q}=H(E_{p,q})
$$
with differentials being 
$H(d_{p,q}):H(E_{p-1,q}) \rightarrow H(E_{p,q})$.

\bigskip

\begin{df}
Let $\T$ be a triangulated category and $X$ an 
object of $\T$.
\begin{enumerate}
\item A tower $X_\bullet$ over $X$ is the data of a sequence
$(X_p \rightarrow X)_{p \in \ZZ}$ of objects over $X$
and a sequence of morphisms over $X$
$$
.. \rightarrow X_{p-1} \xrightarrow{j_p} X_p
 \rightarrow ...
$$
\item Let $X_\bullet$ be a tower over $X$. Suppose
that for each integer $p$ 
we are given a distinguished triangle
$$
X_{p-1} \xrightarrow{j_p} X_{p}
 \xrightarrow{\pi_p} C_p
 \xrightarrow{\delta_p} X_p[1]
$$
where $j_p$ is the structural morphism of the tower
$X_\bullet$.

Then we associate with the tower $X_\bullet$ and the choice
of cones $C_\bullet$ a triangulated exact couple
$$
D_{p,q}=X_p[-p-q], \qquad
E_{p,q}=C_p[-p-q]
$$
with structural morphisms
$$
\alpha_{p,q}=j_p[-p-q], \
\beta_{p,q}=\pi_{p}[-p-q], \
\gamma_{p,q}=\delta_{p}[-p-q].
$$
\end{enumerate}
\end{df}

Let $H:\T^{op} \rightarrow \A$ be a cohomological functor.
In the situation of this definition, we thus have a spectral 
sequence of $E_1$-term:
$E_1^{p,q}=H^{p+q}(C_p)$.

We consider the case where $X_\bullet$ is bounded and
exhaustive \emph{i.e.}
$$
X_p=\left\{ \begin{array}{ll} 0 & \text{if } p \ll 0 \\
X & \text{if } p \gg 0. \end{array} \right.
$$

In this case, the spectral sequence is concentrated
in a band with respect to $p$ and
we get a convergent spectral sequence
$$
E_1^{p,q}=H^{p+q}(C_p) 
 \Rightarrow H^{p+q}(X).
$$
The filtration on the abutment is then given by the
formula
$$
Filt^r(H^{p+q}(X))=\mathrm{Ker}\left(
H^{p+q}(X) \rightarrow H^{p+q}(X_r)
\right).
$$


\subsubsection{Definition}

In the next definitions,
 we introduce our main example of an exact couple,
 based on a filtration by certain open subsets.
\begin{df}
Let $X$ be a scheme.

A \emph{flag} on $X$ is a decreasing sequence $(Z^p)_{p \in \NN}$ 
of closed subschemes of $X$ such that for all integer $p\geq 0$,
$Z^p$ is of codimension greater or equal to $p$ in $X$.
We let $\drap X$ be the set of flags of $X$, 
ordered by termwise inclusion.
\end{df}

We will consider a flag $(Z^p)_{p \in \NN}$ has a 
$\ZZ$-sequence by putting $Z^p=X$ for $p<0$.
It is an easy fact that, with the above definition,
 $\drap X$ is right filtering.

Recall that a pro-object of a category $\mathcal C$ 
is a (\emph{covariant}) functor $F$ from a left filtering 
category $\mathcal I$ to the category $\mathcal C$.
Usually, we will denote $F$ by the intuitive notation
$\pplim{i \in \mathcal I} F_i$
 and call it the \emph{formal projective limit}.
%
\begin{df}
Let $X$ be a scheme.
We define the \emph{coniveau filtration} of $X$
as the sequence $(F_pX)_{p \in \ZZ}$ of pro-open subschemes 
of $X$ such that~:
$$
F_pX=\pplim{Z^* \in \drap X^{op}} (X-Z^p).
$$

We denote by $j_p:F_{p-1}X \rightarrow F_{p}X$ the canonical 
pro-open immersion,
$$
j_p=\pplim{Z^* \in \drap X^{op}} 
\left((X-Z^{p-1}) \rightarrow (X-Z^{p})\right).
$$
\end{df}

\begin{rem}
Usually, the coniveau filtration occurs on cohomology.
 As we will see below, the filtration we have just introduced on
 a scheme will give rise to the classical coniveau filtration, on cohomology.
 Therefore, we think our terminology is both handy and sufficiently accurate.
\end{rem}

Unfortunately, this is a filtration by pro-schemes, and 
if we apply to it the functor $M$ termwise, we
obtain a filtration of $\mot X$ in the category 
$\pro{\dmgme}$.
This latter category is never triangulated. Nonetheless, 
the definition of an exact couple still makes sense 
for the pro-objects of a triangulated category if we replace
distinguished triangles by pro-distinguished triangles\footnote{
\emph{i.e.} the formal projective limit of distinguished triangles.}.
We consider the tower of pro-motives above
the constant pro-motive $\mot X$
$$
... \rightarrow \mot{F_{p-1}X} \xrightarrow{j_{p*}} \mot{F_{p}X}
\rightarrow ...
$$
We define the following canonical pro-cone
$$
Gr_p^M(X)=\pplim{Z^* \in \drap X^{op}} \mot{X-Z^p/X-Z^{p-1}}.
$$
using relative motives -- see the general notations of the article.
We thus obtain pro-distinguished triangles:
$$
\mot{F_{p-1}X} \xrightarrow{j_{p*}} \mot{F_{p}X}
 \xrightarrow{\pi_p} Gr_p^M(X)
  \xrightarrow{\delta_p} \mot{F_{p-1}X}[1].
$$
\begin{df}
\label{coniv_exact_couple}
Consider the above notations.
We define the motivic coniveau 
exact couple associated with $X$ in $\pro{\dmgme}$ as
$$
D_{p,q}=\mot{F_pX}[-p-q], \ \qquad
E_{p,q}=Gr_p^M(X)[-p-q],
$$
with structural morphisms
$$
\alpha_{p,q}=j_p[-p-q], \ 
\beta_{p,q}=\pi_{p}[-p-q], \ 
\gamma_{p,q}=\delta_{p}[-p-q].
$$
\end{df}
According to the notation which follows Definition \ref{triang_e_couple},
 the differential associated with the motivic coniveau exact couple is
 equal to the composite map of the following diagram:
\begin{equation} \label{eq:df:diff}
\begin{split}
\xymatrix@R=20pt@C=27pt{
Gr_{p+1}^M(X)[-p-q-1]\ar^-{\delta_{p+1}}[r]\ar_/-45pt/{d_{p+1,q}}[rrd]
 & \mot{F_{p}X}[-p-q]\ar@{=}[d] & \\
 &  \mot{F_{p}X}[-p-q]\ar_-{\pi_{p}}[r] & Gr_{p}^M(X)[-p-q].
}
\end{split}
\end{equation}

%

\subsection{Computations} \label{sec:computations}

\subsubsection{Recollection and complement on generic motives}

\label{sec:generic_motives}

\begin{num} \label{num:model}
We will call {\it function field} any
 finite type field extension $E/k$.
A {\it model} of the function field $E$ will be a connected smooth 
scheme $X/k$ with a given $k$-isomorphism between the function field 
of $X$ and $E$.
Recall the following definition from \cite[3.3.1]{Deg5bis}~:
\end{num}
\begin{df}
Consider a function field $E/k$ and an integer $n \in \ZZ$. 
We define the {\it generic motive of $E$ with weight $n$} as
the following pro-object of $\dmgm$~:
$$
M(E)(n)[n]:=\pplim{A \subset E, \ \spec A \text{ model of } E/k} M(\spec A)(n)[n].
$$
We denote by $\dmgmo$ the full subcategory of $\mathrm{pro}\!-\!\dmgm$ consisting
of the generic motives.
\end{df}
Of course, given a function field $E$ with model $X/k$,
 the pro-object $M(E)$ is canonically isomorphic to the pro-motive
  made by the motives of non empty open subschemes of $X$.

\begin{num} \label{functoriality_generic_motive}
The interest of generic motives lies in their functoriality
which we now review~: \\
(1) Given any extension of function fields
 $\varphi:E \rightarrow L$,
  we get a morphism $\varphi^*:M(L) \rightarrow M(E)$ 
   (by covariant functoriality of motives). \\
(2) Consider a finite extension of function fields
 $\varphi:E \rightarrow L$. One can find respective models $X$ and $Y$
  of $E$ and $L$ together with a finite morphism of schemes $f:Y \rightarrow X$
   which induces on function fields the morphism $\varphi$ through the structural
    isomorphisms. \\
For any open subscheme $U \subset X$,
 we put $Y_U=Y \times_X U$ and let $f_U:Y_U \rightarrow U$
be the morphism induced by $f$. It is finite and surjective. In particular,
its graph seen as a cycle in $U \times Y_U$ defines a finite correspondence
from $U$ to $Y_U$,
denoted by $\tra f_U$ and called the \emph{transpose} of $f_U$.
We define the {\it norm morphism} $\varphi_*:M(E) \rightarrow M(L)$ as
 the well defined pro-morphism (see \cite[5.2.9]{Deg5bis})
$$
\pplim{U \subset X} \Bigg(M(U) \xrightarrow{(\tra f|_U)_*} M(Y_U)\Bigg)
$$
through the structural isomorphisms of the models $X$ and $Y$. \\
(3) Consider a function field $E$ and a unit $x \in E^\times$.
Given a smooth sub-$k$-algebra $A \subset E$ which contains $x$ and $x^{-1}$,
we get a morphism $\spec A \rightarrow \GG$.
Let us denote by $\gamma_x^A:\spec A \rightarrow \spec A \times \GG$ its
 graph.
Recall the canonical decomposition $M(\GG)=\ZZ \oplus \ZZ(1)[1]$ 
and consider the associated projection $M(\GG) \xrightarrow \pi \ZZ(1)[1]$. 
We associate with the unit $x$
a morphism $\gamma_x:M(E) \rightarrow M(E)(1)[1]$ defined as follows:
$$
\pplim{x,x^{-1} \in A \subset E} \big( M(\spec A)
 \xrightarrow{\gamma_{x*}^A} M(\spec A) \otimes M(\GG)
 \xrightarrow{\pi} M(\spec A)(1)[1]\big).
$$
One can prove moreover that if $x \neq 1$,
$\gamma_x \circ \gamma_{1-x}=0$ and $\gamma_{1-x} \circ \gamma_x=0$
so that any element $\sigma \in K_n^M(E)$ of Milnor K-theory defines
a morphism $\gamma_\sigma:M(E) \rightarrow M(E)(n)[n]$
 (see \cite[5.3.5]{Deg5bis}). \\
(4) Let $E$ be a function field and $v$ a discrete valuation on $E$
with ring of integers $\mathcal O_v$ essentially of finite type over $k$.
Let $\kappa(v)$ be the residue field of $v$. \\
As $k$ is perfect, there exists a connected smooth scheme $X$ with a 
point $x \in X$ of codimension $1$ such that
 $\mathcal O_{X,x}$ is isomorphic to $\mathcal O_v$. 
This implies $X$ is a model of $E/k$. Moreover,
reducing $X$, one can assume the closure $Z$ of $x$ in $X$ is smooth
so that it becomes a model of $\kappa(v)$. \\
For an open neighborhood $U$ of $x$ in $X$,
 we put $Z_U=Z \times_X U$.
We define the {\it residue morphism} 
$\partial_v:M(\kappa(v))(1)[1] \rightarrow M(E)$ associated with $(E,v)$ 
as the pro-morphism
$$
\pplim{x \in U \subset X}
 \big(M(Z_U)(1)[1] \xrightarrow{\partial_{U,Z_U}} M(U-Z_U)\big).
$$
The fact it is a morphism of pro-objects follows
 from the functoriality of residues with respect to open immersions
 (see \cite[5.4.6]{Deg5bis}).
 \end{num}

\begin{rem}
These morphisms satisfy a set of relations 
which in fact corresponds exactly to the axioms of a cycle premodule 
by M.~Rost (cf \cite[(1.1)]{Ros}). 
We refer the reader to \cite[5.1.1]{Deg5bis} for a precise statement.
\end{rem}

\begin{num}
Consider again the situation and notations of point (2) 
in paragraph \ref{functoriality_generic_motive}.
With the Gysin morphism we have introduced before,
one can give another definition for the norm morphism of
generic motives. \\
Indeed, for any open subscheme $U$ of $X$, the morphism 
$f_U:Y_U \rightarrow U$ is finite of relative dimension $0$ and 
thus induces a Gysin morphism $f_U^*:M(U) \rightarrow M(Y_U)$ 
 -- see \eqref{eq:gysin}. The morphism $f_U^*$ is natural
 with respect open immersions according to \cite[Prop. 2.10]{Deg6}. 
 Thus, we get a morphism of pro-objects
$$
\pplim{U \subset X}
 \big( \mot U \xrightarrow{f_U^*} \mot{Y_U)} \big).
$$
which induces through the structural isomorphisms of the
models $X$ and $Y$ a morphism 
$\varphi'_*:\mot E \rightarrow \mot L$.
\end{num}
\begin{lm}
Consider the above notations.
Then, $\varphi'_*=\varphi_*$.
\end{lm}
\begin{proof}
By functoriality, we can restrict the proof to 
the cases where $L/E$ is separable or $L/E$ is purely inseparable.

In the first case, we can choose a model $f:Y \rightarrow X$ 
of $\varphi$ which is {\'e}tale. Then the lemma follows from
\cite[Prop. 2.15]{Deg6}.

In the second case, we can assume that $L=E[\sqrt[q] a]$ for
$a \in E$. Let $A \subset E$ be a sub-$k$-algebra containing 
$a$ such that $X=\spec A$ is a smooth scheme. Let
$B=A[t]/(t^q-a)$. Then $Y=\spec B$ is again a smooth scheme
(over $k$)
and the canonical morphism $f:Y \rightarrow X$ is a model
of $L/E$. We consider its canonical factorisation
$Y \xrightarrow i \PP^1_X \xrightarrow p X$ corresponding
to the parameter $t$, together with the following diagram
 made of two cartesian squares:
$$
\xymatrix@R=12pt@C=22pt{
Y \times_X Y\ar^/4pt/j[r]\ar[d]
 & {\PP^1_Y}\ar^/-1pt/{f'}[d]\ar^q[r]
 & Y\ar^f[d] \\
Y\ar^i[r] & {\PP^1_X}\ar^p[r] & X. \\
}
$$
The scheme $Y \times_X Y$ is non reduced and its 
reduction is $Y$. Moreover, the canonical immersion 
$Y \rightarrow Y \times_X Y$ is an exact thickening of order
$q$ in $Y$ according to the definition of \cite[Par. 1.18]{Deg6}.
Thus, the following diagram is commutative~:
$$
\xymatrix@R=14pt@C=20pt{
\mot Y\ar@{}|/-2pt/{_{(1)}}[rd]
 & \mot {\PP^1_Y}\ar_{j^*}[l]\ar@{}|/-2pt/{_{(2)}}[rd]
 & \mot Y\ar_/-4pt/{q^*}[l] \\
 \mot Y\ar@{=}[u]
 & \mot {\PP^1_X}\ar_/-2pt/{\tra f'_*}[u]\ar_{i^*}[l]
 & \mot X.\ar_{\tra f_*}[u]\ar_/-4pt/{p^*}[l] \\
}
$$
Indeed, part (2) (resp. (1)) is commutative by \cite[2.2.15]{Deg5bis}
(resp. \cite[2.5.2: (2)]{Deg5bis}). Thus $f^*=\tra f_*$ and this concludes.
\end{proof}


\subsubsection{The graded terms}

For a scheme $X$, we denote by $X^{(p)}$ the set
of points of $X$ of codimension $p$. If $x$ is a point of $X$,
$\kappa(x)$ will denote its residue field. The symbol $\pprod{}$ 
denotes the product in the category of pro-motives.
\begin{lm}
\label{terms_coniveau}
Let $X$ be a smooth scheme and consider the notations
 of Definition \ref{coniv_exact_couple}.
Then, for all integer $p \geq 0$,
 purity isomorphisms of the form \eqref{eq:purity} induce a canonical isomorphism
$$
Gr_p^M(X) \xrightarrow{\epsilon_p} \pprod{x \in X^{(p)}} \mot{\kappa(x)}(p)[2p].
$$
\end{lm}
In particular, for any point $x \in X^{(p)}$ we get a canonical projection map:
\begin{equation} \label{eq:proj_e_c}
\pi_x:Gr_p^M(X) \rightarrow \mot{\kappa(x)}(p)[2p].
\end{equation}
\begin{proof}
Let $\mathcal I_p$ be the set of pairs $(Z,Z')$ such that
$Z$ is a reduced closed subscheme of $X$ of codimension $p$ and
$Z'$ is a closed subset of $Z$ containing its singular 
locus. Then
$$
Gr_p^M(X) \simeq \pplim{(Z,Z') \in \mathcal I_p} \mot{X-Z'/X-Z}.
$$
For any element $(Z,Z')$ of $\mathcal I_p$, 
under the purity isomorphism \eqref{eq:purity}, we get: \\
$\mot{X-Z'/X-Z} \simeq \mot{Z-Z'}(p)[2p]$.

For any point $x$ of $X$, we let $Z(x)$ be the reduced 
closure of $x$ in $X$ and
$\mathcal F(x)$ be the set of closed subschemes $Z'$ of $Z(x)$ 
 containing the singular locus $Z(x)_{sing}$ of $Z(x)$.
By additivity of motives,
 we finally get an isomorphism:
$$
Gr_p^M(X) \simeq \pprod{x \in X^{(p)}} 
 \pplim{Z' \in \mathcal F(x)} \mot{Z(x)-Z'}(p)[2p].
$$
This implies the lemma because $Z(x)-Z(x)_{sing}$ is
a model of $\kappa(x)$.
\end{proof}

\subsubsection{The differentials}

\num \label{df:differentielles}
Let $X$ be a scheme essentially of finite type\footnote{
For the purpose of the next proposition, we need only the case
where $X$ is smooth but the general case treated here will be
used later.} over $k$
 and consider a couple $(x,y) \in X^{(p)} \times X^{(p+1)}$.

Assume that $y$ is a specialisation of $x$.
Let $Z$ be the reduced closure of $x$ in $X$
 and $\tilde Z \xrightarrow f Z$ be its normalisation. 
Each point $t \in f^{-1}(y)$ corresponds to a
discrete valuation $v_t$ on $\kappa(x)$ with residue 
field $\kappa(t)$.
We denote by $\varphi_t:\kappa(y) \rightarrow \kappa(t)$ the 
morphism induced by $f$. Then, we define the following morphism 
of generic motives
\begin{equation} \label{eq:residues_pt}
\partial_y^x=
\sum_{t \in f^{-1}(y)}
 \partial_{v_t} \circ \varphi_{t*}:
M(\kappa(y))(1)[1]
  \rightarrow M(\kappa(x))
\end{equation}
using the notations of \ref{functoriality_generic_motive}. \\
If $y$ is not a specialisation of $x$,
  we put conventionally $\partial_y^x=0$.
\begin{prop}
\label{compute_differentials}
Consider the above hypothesis and notations.
If $X$ is smooth then the following diagram is commutative:
$$
\xymatrix@C=45pt{
Gr_{p+1}^M(X)\ar^{d_{p+1,-p-1}}[r]\ar_{\pi_y}[d]
 & Gr_p^M(X)[1]\ar^{\pi_x}[d] \\
M(\kappa(y))(p+1)[2p+2]
 \ar^-{\partial_y^x}[r]
  & M(\kappa(x))(p)[2p+1]
}
$$
where the vertical maps are defined in \eqref{eq:proj_e_c}
 and $d_{p+1,-p-1}$ in \eqref{eq:df:diff}.
\end{prop}
Of course, 
 this proposition determines every differentials of the motivic
 coniveau exact couple as $d_{p,q}=d_{p,-p}[-p-q]$.
\begin{proof}
According to Definition \ref{coniv_exact_couple},
 the morphism $d_{p+1,-p-1}$ is the formal projective limit of the morphisms
\begin{equation} \label{eq:pf:compute_differentials}
\mot{X-W/X-Y} \rightarrow \mot{X-Y}[1] 
\rightarrow \mot{X-Y/X-Z}[1],
\end{equation}
for large enough closed subsets $W \subset Y \subset Z$ of $X$
such that $\mathrm{codim}_X(Z)=p$, $\mathrm{codim}_X(Y)=p+1$
 and $\mathrm{codim}_X(W) \geq p+2$.
For the proof, we will consider $W \subset Y \subset Z$ as above,
 assume that $y \in Y$, $x \in Z$
 and study \eqref{eq:pf:compute_differentials} for
 $Z$, $Y$, $W$ large enough. 
To simplify the notations,
 we will replace $X$ by $X-W$ which means practically that 
 we can substract any subset of $X$ if it has codimension greater
 than $p+1$.

First of all, enlarging $Y$,
 we can assume that it contains the singular locus of $Z$. 
Because the singular locus of $Y$ has codimension greater than $p+1$
 in $X$, we can assume by reducing $X$ that $Y$ is smooth.
 Then, using the purity isomorphism,
  the composite map \eqref{eq:pf:compute_differentials}
 is isomorphic to the following one:
\begin{equation*} \label{eq2:pf:compute_differentials}
M(Y)\,\dtwist{p+1}
 \xrightarrow{\ \partial_{X,Y}\ } M(X-Y)[1] 
 \xrightarrow{\ i_Y^* \ } M(Z-Y)\,\dtwist p[1]
\end{equation*}
where $i_Y:(Z-Y) \rightarrow (X-Y)$ is
 the obvious restriction of the canonical closed immersion
 $i:Y \rightarrow Z$.

Let $Y_y$ (resp. $Z_x$) be the irreducible component
 of $Y$ (resp. $Z$) containing $y$ (resp. $x$). 
As $Y$ is smooth, we can write $Y=Y_y \sqcup Y'_y$
 where $Y'_y$ is the complement of $Y_y$ in $Y$.
As $(Z-Y)$ is smooth, if we put $\hat Y_x=Y \times_Z Z_x$
 then $(Z_x-\hat Y_x)$ is a connected component of $(Z-Y)$.
 We denote by $i_x:(Z_x-\hat Y_x) \rightarrow (X-Y)$
 the obvious restriction of $i_Y$.
According to \cite[Prop. 1.36]{Deg6},
 the following diagram is commutative:
$$
\xymatrix@C=44pt{
M(Y)\,\dtwist{p+1}\ar^-{\partial_{X,Y}}[r]\ar@{->>}[d]
 & M(X-Y)[1]\ar^-{i_Y^*}[r]\ar@{=}[d]
 & M(Z-Y)\,\dtwist p[1]\ar@{->>}[d] \\
M(Y_y)\,\dtwist{p+1}\ar^-{\partial_{X-Y'_y,Y_y}}[r]
 \ar@/_10pt/_{\partial_{Y,y}^{Z,x}}[rr]
 & M(X-Y)[1]\ar^-{i_x^*}[r]
 & M(Z_x-\hat Y_x)\,\dtwist p[1]
}
$$
where the vertical maps are the canonical projections.
The proposition is equivalent to show
 that the formal projective limit of the maps
 $\partial_{Y,y}^{Z,x}$ for $Z$, $Y$, $W$ large enough
 is equal to $\partial_y^x$
 (remember we have identified $X$ with $X-W$).

\bigskip

Assume that $y$ is not a specialisation of $x$. 
 Then $Y_y \cap Z_x$ has codimension greater 
 than $p+1$ in $X$. Therefore, reducing $X$ again,
 we can assume $Y_y \cap Z_x=\emptyset$.
 A fortiori, $Y_y \cap (Z_x-\hat Y_x)=\emptyset$
 and we get the following cartesian square of closed immersions:
$$
\xymatrix@=20pt{
\emptyset\ar[r]\ar[d]
 & Y_y\ar^{}[d]  \\
(Z_x-\hat Y_x)\ar^-{}[r] & (X-Y'_y).
}
$$
Then, according to the naturality of residues with respect to
the preceding square (relation (2) of \cite[Th. 1.34]{Deg6}),
 we obtain:
  $\partial_{X-Y'_y,Y_y} \circ i_x^*=0$.
 Thus the proposition is proved in that case.

We now consider the case where $y$ is a specialisation
 of $x$ \emph{i.e.} $Y_y \subset Z_x$.
Then $Y_y \subset \hat Y_x$: to simplify the notation,
 we can assume that $Z=Z_x$ 
 \emph{i.e.} $Z$ is irreducible with generic point $x$.
Let $f:\tilde Z \rightarrow Z$ be the normalization of $Z$.
The singular locus $\tilde Z_{sing}$ of $\tilde Z$
is everywhere of codimension greater than $1$ in $\tilde Z$.
Thus, $f(\tilde Z_{sing})$ is everywhere of codimension 
greater than $p+1$ in $X$,
and we can assume by reducing $X$ again 
that $\tilde Z$ is smooth.

Let us denote by $\tilde Y$ (resp. $\tilde Y_y$, $\tilde Y'_y$)
the reduced inverse image of $Y$ (resp. $Y_y$, $Y'_y$)
along $f$. Reducing $X$ again,
 we can assume that $\tilde Y_y$ is smooth
 and $\tilde Y_y \cap \tilde Y'_y=\emptyset$.
 Moreover,
 we can assume that every connected component of $\tilde Y_y$
 dominates $Y_y$ (by reducing $X$, we can throw away the non dominant
 connected components).
 In other words, the map $g_y:\tilde Y_y \rightarrow Y_y$
 induced by $f$ is finite and equidimensional.
 Then we can consider the following topologically cartesian square:
$$
\xymatrix@=18pt{
\tilde Y_y\ar^-{\tilde \sigma}[r]\ar_{g_y}[d]
 & (\tilde Z-\tilde Y'_y)\ar[d] \\
Y_y\ar^-\sigma[r] & (X-Y'_y)
}
$$
where $\sigma$ and $\tilde \sigma$ are the obvious closed immersions
 and the right vertical map is induced by the composite map
$\tilde Z \xrightarrow f Z \xrightarrow i X$. 
Note that taking the respective complements
 of $\tilde \sigma$ and $\sigma$
 in the source and target of this composite map,
 it induces the following one:
$$
(\tilde Z-\tilde Y) \xrightarrow h (Z-Y) \xrightarrow i (X-Y).
$$
Thus,
 applying the naturality of residues with respect to Gysin morphisms
 (\cite[Prop. 2.13]{Deg6}) to the preceding square on the one hand
 and the functoriality of the Gysin morphism (\cite[Prop. 2.9]{Deg6})
 on the other hand,
 we obtain the following commutative diagram:
$$
\xymatrix@R=18pt@C=45pt{
\mot{Y_y}\dtwist{p+1}\ar_-{\partial_{X-Y'_y,Y_y}}[r]\ar@{=}[d]
  \ar@/^12pt/^{\partial_{Y,y}^{Z,x}}[rr]
 & \mot{X-Y}[1]\ar_-{i^*}[r] 
 & \mot{Z-Y}\dtwist p[1]\ar^-{h^*}[d] \\
\mot{Y_y}\dtwist{p+1}\ar_-{g_y^*}[r]
 &\mot{\tilde Y_y}\dtwist{p+1}\ar_-{\partial_{\tilde Z-\tilde Y'_y,\tilde Y_y}}[r]
 & \mot{\tilde Z-\tilde Y}\dtwist p[1].
}
$$
Note that the set of connected components
 of the smooth scheme $\tilde Y_y$
 corresponds bijectively to the set $f^{-1}(y)$. For any $t \in f^{-1}(y)$,
 we denote by $\tilde Y_t$ the corresponding connected component
 so that $\tilde Y_y=\sqcup_{t \in f^{-1}(y)} \tilde Y_t$.
 Note that $\tilde Y_t$ is also a connected component of $\tilde Y$.
 We put:
$$
\tilde Z_t=\tilde Z-(\tilde Y-\tilde Y_t).
$$
This is an open subscheme of $\tilde Z$ containing $\tilde Y_t$ and
 $(\tilde Z_t-\tilde Y_t)=(\tilde Z-\tilde Y)$.
 Finally applying the additivity properties of Gysin morphisms
  and residues
 (\cite[Prop. 1.36]{Deg6}),
 we obtain the following commutative squares:
$$
\xymatrix@R=18pt@C=45pt{
\mot{Y_y}\dtwist{p+1}\ar^-{g_y^*}[r]\ar@{=}[d]
 &\mot{\tilde Y_y}\dtwist{p+1}\ar^-{\partial_{\tilde Z-\tilde Y'_y,\tilde Y_y}}[r]
 & \mot{\tilde Z-\tilde Y}\dtwist p[1]\ar@{=}[d] \\
\mot{Y_y}\dtwist{p+1}\ar^-{\sum_t g_t^*}[r]\ar@/_16pt/_{\tilde \partial_{Y,y}^{Z,x}}[rr]
 & \underset{t \in f^{-1}(y)}{\bigoplus}\mot{\tilde Y_t}\dtwist{p+1}
    \ar_\sim[u]
  \ar^-{\sum_t \partial_{\tilde Z_t,\tilde Y_t}}[r]
 & \mot{\tilde Z-\tilde Y}\dtwist p[1]
}
$$
where the middle vertical map is the canonical isomorphism.
We can now identify $\partial_y^x$
 with the formal projective limit of $\tilde \partial_{Y,y}^{Z,x}$
 for $Y$, $W$ large enough (remember we have assumed $Z=Z_x$).
 In view of formula \eqref{eq:residues_pt}, this is justified because:
\item[\indent -] $h$ is birational and $(\tilde Z-\tilde Y)$ is
 a smooth model of $\kappa(x)$.
\item[\indent -] The closed pair $(\tilde Z_t,\tilde Y_t)$
 is smooth of codimension $1$ and the local ring of $\mathcal O_{\tilde Z_t,\tilde Y_t}$
 is isomorphic (through $h$) to the valuation ring $\mathcal O_{v_t}$
 corresponding to the valuation $v_t$ on $\kappa(x)$
 considered in Paragraph \ref{df:differentielles}.
%
\end{proof}


\section{Cohomological realization}

We fix a Grothendieck abelian category $\A$
and consider a cohomological functor
$$
H:\dmgm^{op} \rightarrow \A,
$$
simply called a {\it realization functor}.
To such a functor,
 we associate a twisted cohomology theory:
 for a smooth scheme $X$ 
  and any pair of integers $(i,n) \in \ZZ^2$, we put:
$$
H^{i,n}(X)=H\big( \mot X(-n)[-i] \big).
$$
By the very definition, this functor is contravariant,
not only with respect to morphisms of smooth schemes but
also for finite correspondences. 
 Using Gysin morphisms \eqref{eq:gysin},
  it is also covariant with respect to projective morphisms.

\subsection{The coniveau spectral sequence}

\begin{num}
The functor $H$ admits an obvious extension to pro-objects
\begin{equation} \label{H_extended_pro}
\bar H:\pro{\dmgm}^{op} \rightarrow \A
\end{equation}
which sends pro-distinguished triangles to long exact 
sequences since right filtering colimits are exact in $\A$.
Moreover, for any function fields $E/k$, we simply put
\begin{equation} \label{eq:unramified_cycle_module}
\hat H^{i,n}(E):=\ilim{A \subset E} H^{i,n}(\spec A)
\end{equation}
where the limit is taken over the models of the function field $E/k$
 (see Paragraph \ref{num:model}).

Fix an integer $n \in \ZZ$.
We apply the functor $\bar H(?(n))$
to the pro-exact couple of Definition \ref{coniv_exact_couple}.
We then obtain a converging spectral sequence 
which, according to Lemma \ref{terms_coniveau}, has the form:
\begin{equation}
\label{coniv_spectral}
E_1^{p,q}(X,n)=\bigoplus_{x \in X^{(p)}}
 \hat H^{q-p,n-p}(\kappa(x))
  \Rightarrow H^{p+q,n}(X).
\end{equation}
This is a spectral sequence which converges to the so-called 
 \emph{coniveau filtration} on the twisted cohomology $H^{**}$
 defined by:
\begin{equation} \label{eq:coniveau_filtration}
N^rH^{i,n}(X)= \bigcap_{U \subset X, \ \mathrm{codim}_X(X-U) \geq r}
\mathrm{Ker}\left(H^{i,n}(X) \rightarrow H^{i,n}(U)\right).
\end{equation}
We also call the above spectral sequence the coniveau spectral sequence
 of $X$ with coefficients in $H$.
\end{num}


\subsection{Cycle modules}

\label{cycle_modules}

We now recall below the theory of Rost cycle modules
 in a way adapted to our needs.

\num The first step in Rost's theory is the notion of a
 {\it cycle premodule}. As already mentioned in the introduction,
 it is a covariant functor from the category of function fields to the category of
 graded abelian groups satisfying an enriched functoriality exactly analog 
 to that of Milnor K-theory $K_*^M$.
In our context, we will define\footnote{Indeed, when $\A$ is the category
 of abelian groups, it is proved in \cite[th. 5.1.1]{Deg5bis} 
 that such a functor defines a cycle premodule in the sense of M.~Rost.}
 a cycle premodule as a functor
$$
\phi:\dmgmo^{op} \rightarrow \A.
$$
Usually, we put
\begin{equation} \label{eq:grad_cycle-module}
\phi(M(E)(-n)[-n])=\phi_n(E)
\end{equation}
 so that $\phi$ becomes
a graded functor on function fields.
In view of the description of the functoriality of generic
motives recalled in \ref{functoriality_generic_motive},
 $\phi$ is equipped with the following structural maps:
\begin{enumerate}
\item For any extension of function fields, $\varphi:E \rightarrow L$,
 a \emph{corestriction} $\varphi_*:\phi_*(E) \rightarrow \phi_*(L)$ of  degree $0$.
\item For any finite extension of function fields, $\varphi:E \rightarrow L$,
 a \emph{restriction} $\varphi^*:\phi_*(L) \rightarrow \phi_*(E)$ of  degree $0$,
  also denoted by $N_{L/E}$.
\item For any function field $E$, $\phi_*(E)$ admits a $K_*^M(E)$-graded 
module structure.
\item For any valued function field $(E,v)$ with ring of integers 
essentially of finite type over $k$ and residue field $\kappa(v)$,
 a \emph{residue} $\partial_v:\phi_*(E) \rightarrow \phi_*(\kappa(v))$ of degree $-1$.
\end{enumerate}

\begin{df} \label{generic_transform4real}
For any pair of integers $(q,n)$, we associate with the realization functor $H$
 a cycle premodule $\hat H^{q,n}$ 
 as the restriction of the functor $\bar H(?(n)[p])$ to the category $\dmgmo$,
 using notation \eqref{H_extended_pro}.
\end{df}
According to formula \eqref{eq:grad_cycle-module},
 cycle modules are $\ZZ$-graded. 
 This motivates the following redundant notation
 for a smooth scheme $X$ and a triple of integers $(q,n,r)$:
\begin{equation} \label{eq:Rost_grading_cohomology}
H^{q,n}_r(X):=H^{q+r,n+r}(X).
\end{equation}
Note that we thus get a twisted cohomology $H_*^{**}$ with values
 in the $\ZZ$-graded category $\A^\ZZ$.
 Moreover, we obtain using the motivic coniveau exact couple 
 a spectral sequence of the form:
\begin{equation}
\label{coniv_spectral_graded}
E_1^{p,q}(X,n)=\bigoplus_{x \in X^{(p)}}
 \hat H^{q-p,n-p}_*(\kappa(x))
  \Rightarrow H^{p+q,n}_*(X),
\end{equation}
with values in $\A^\ZZ$.
Its $0$-th graded part is 
 the spectral sequence \eqref{coniv_spectral_graded}. \\
For any function field $E$ and any triple of integers $(q,n,r)$,
 we accordingly write:
\begin{equation} \label{Rost_grading_unramified_cycle_module}
\hat H^{q,n}_{r}(E)=\hat H^{q+r,n+r}(E)
\end{equation}
with the notation of formula \eqref{eq:unramified_cycle_module}.

\begin{rem} \label{rem:relation_generic}
Taking care of the canonical grading on cycle modules,
 the family of cycle modules defined above
 comes in with the following relation:
\begin{equation*} \label{relation_generic}
\forall a \in \ZZ, \hat H^{q+a,n+a}_{*}=\hat H^{q,n}_{*-a}.
\end{equation*}
\end{rem}

\num \label{num:cycle_complex}
Recall the aim of the axioms of a cycle module is to build a complex
 (cf \cite[(2.1)]{Ros}).
We recall these axioms to the reader using the morphisms
 introduced in Paragraph \ref{df:differentielles}.
We say that a cycle premodule $\phi$ is a {\it cycle module}
 if the following two conditions are fulfilled~:
\begin{enumerate}
\item[(FD)] Let $X$ be a normal scheme essentially of finite type over $k$,
 $\eta$ its generic point and $E$ its functions field.
Then for any element $\rho \in \phi_i(E)$,
 $\phi(\partial_x^\eta)(\rho)=0$
for all but finitely many points $x$ of codimension $1$ in $X$.
\item[(C)] Let $X$ be an integral local scheme essentially of finite 
type over $k$ and of dimension $2$.
Let $\eta$ (resp. $s$) be its generic (resp. closed) point,
 and $E$ (resp. $\kappa$) be its function (resp. residue) field.
Then, for any integer $n \in \ZZ$, the morphism
$$
\sum_{x \in X^{(1)}} \phi_{n-1}(\partial_s^x) \circ \phi_n(\partial_x^\eta):
 \phi_n(E) \rightarrow \phi_{n-2}(\kappa),
$$
well defined under (FD), is zero.
\end{enumerate}
When these conditions are fulfilled, 
for any scheme $X$ essentially of finite type over $k$,
Rost defines in
\cite[(3.2)]{Ros} a graded complex of
{\it cycles with coefficients in $\phi$} whose $i$-th graded\footnote{
This graduation follows the convention of \cite[\S 5]{Ros} except
for the notation. The notation $C^p(X;\phi,i)$ used by Rost would introduce
a confusion with twists.}
  $p$-cochains are
\begin{equation} \label{eq:C^p_modl}
C^p(X;\phi)_i=\bigoplus_{x \in X^{(p)}} \phi_{i-p}(\kappa(x))
\end{equation}
and with $p$-th differential equal to the well defined morphism
\begin{equation} \label{eq:diff_modl}
d^p=\sum_{(x,y) \in X^{(p)} \times X^{(p+1)}} \phi(\partial_x^y).
\end{equation}
The cohomology groups of this complex are called the 
\emph{Chow groups with coefficients in $\phi$}
and denoted by $A^*(X;\phi)$ in \cite{Ros}.
Actually, $A^*(X;\phi)$ is bigraded according to the bigraduation on $C^*(X;\phi)$.

\begin{num} \label{num:prelim:coniveau&Rost}
Consider the cycle module $\hat H_*^{q,n}$ introduced
 in Definition \ref{generic_transform4real} with
 its $\ZZ$-graduation given by formula \eqref{eq:grad_cycle-module}.
 According to this definition, the $E_1$-term of the spectral sequence
 \eqref{coniv_spectral_graded} can be written as:
$$
E_1^{p,q}(X,n)=C^p(X,\hat H_*^{q,n})
$$
if we use the formula \eqref{eq:C^p_modl} for the right hand side.
Moreover, according to Proposition \ref{compute_differentials},
 the differential $d_1^{p,q}$ of the spectral sequence is given
 by the formula:
$$
d_1^{p,q}=\sum_{(x,y) \in X^{(p)} \times X^{(p+1)}}
 \hat H_*^{q,n}(\partial_y^x).
$$
This is precisely the formula \eqref{eq:diff_modl}
 for the cycle premodule $\hat H^{q,n}$.
 Thus proposition \emph{loc. cit.} implies in particular
 that this morphism is well defined.
 Moreover it shows that for any integer $r$,
 the graded abelian group $C^*(X,\hat H^{q,n}_r)$
 together with the well defined differentials of the form \eqref{eq:diff_modl}
 is a complex. We deduce from this fact the following proposition:
\end{num}
\begin{prop} \label{realisation:coniveau&Rost}
Consider the previous notations.
\begin{enumerate}
\item[(i)]
For any integer $q \in \ZZ$,
 the cycle premodule $\hat H_*^{q,n}$ is a cycle module.
\item[(ii)]
For any smooth scheme $X$
 and any couple $(q,n)$ of integers,
there is an equality of complexes:
$$
E_1^{*,q}(X,n) = C^*(X;\hat H_*^{q,n}),
$$
where the left hand side is the complex
 made by the $q$-th line of the first page of 
 the spectral sequence \eqref{coniv_spectral_graded}.
\end{enumerate}
\end{prop}
\begin{proof}
Point (ii) follows from Preliminary \ref{num:prelim:coniveau&Rost}.

We prove point (i), axiom (FD). Consider a normal scheme $X$ essentially
 of finite type over $k$. We can assume it is affine of finite type.
Then there exists a closed immersion $X \xrightarrow i \dtenx r k$ for an integer
 $r \geq 0$. 
According to the preliminary discussion of
 Paragraph \ref{num:prelim:coniveau&Rost},
 $C^*(\dtenx r k;\hat H_*^{q,n})$ is a well defined complex.
Thus, axiom (FD) for the cycle premodule $\hat H_*^{q,n}$ follows from the fact 
 the immersion $i$ induces an inclusion
$$
\hat H_*^{q,n}(E) \subset C^r(\dtenx r k;\hat H_*^{q,n})
$$
and the definition of the differentials given above. \\
For axiom (C), we consider an integral local scheme $X$ essentially of finite 
type over $k$ and of dimension $2$. We have to prove that $C^*(X;\hat H_*^{q,n})$ is a 
complex -- the differentials are well defined according to (FD).
To this aim, we can assume $X$ is affine of finite type over $k$. 
Then, there exists a closed immersion $X \rightarrow \dtenx r k$.
From the definition given above, we obtain a monomorphism
$$
C^p(X;\hat H_*^{q,n}) \rightarrow C^p(\dtenx r k;\hat H_*^{q,n})
$$
which is compatible with differentials. Thus the conclusion follows
again from the preliminary discussion of Paragraph \ref{num:prelim:coniveau&Rost}.
\end{proof}
\rem This proposition gives a direct proof of \cite[Th. 6.2.1]{Deg5bis}.

\begin{cor} \label{cor:coniv_spectral2}
Using the notations of the previous proposition,
 the $E_2$-terms of the coniveau spectral sequence \eqref{coniv_spectral_graded} 
  are~:
$$
E_2^{p,q}(X,n)=A^p(X;\hat H_*^{q,n}) \Rightarrow H_*^{p+q,n}(X).
$$
Moreover, for any couple of integers $(q,n)$ and any smooth proper scheme $X$,
 the term $E_2^{0,q}(X,n)$ is a birational invariant of $X$.
\end{cor}
The second assertion follows from \cite[12.10]{Ros}.

\begin{ex} \label{ex:motivic&Milnor}
Consider the functor $H_\M=\Hom_{\dmgm}(.,\ZZ)$,
 corresponding to motivic cohomology.
In this case, following \cite[3.2, 3.4]{SV}, for any function field $E$,
\begin{equation}
H^q_\M(E;\ZZ(p))=
\begin{cases}
0 & \text{if } q>p \text{ or } p<0 \\
K_p^M(E) & \text{if } q=p\geq 0
\end{cases}
\end{equation}
In particular, from Definition \ref{generic_transform4real},
$\hat H^{n,n}_\M=K_{*+n}^M$. 
In fact, this is an isomorphism of cycle modules. 
For the norm, this is \emph{loc. cit.} 3.4.1.
For the residue, it is sufficient (using for example 
\cite[formula (R3f)]{Ros})
to prove that 
for any valued function field $(E,v)$ with uniformizing parameter
$\pi$, $\partial_v(\pi)=1$ for the cycle module $\hat H^{n,n}_\M$.
This follows from \cite[2.6.5]{Deg5bis} as for any
morphism of smooth connected schemes $f:Y \rightarrow X$, 
the pullback $f^*:H^0_\M(X;\ZZ) \rightarrow H^0_\M(Y,\ZZ)$ is the
identity of $\ZZ$.

As remarked by Voevodsky at the very beginning of his theory,
 the vanishing mentioned above implies that the coniveau spectral 
sequence for $H_\M$ satisfies $E_1^{p,q}(X,n)=0$ if $p>n$ or $q>n$.
In particular, the spectral sequence gives an isomorphism:
$A^n(X;\hat H^{n,n})_0 \rightarrow H^{2n}_\M(X;\ZZ(n))$.
The left hand side is $A^n(X;K_*^M)_n$ which is nothing else
 than the Chow group $CH^n(X)$ of cycles modulo rational equivalence.
 This is precisely the proof of the isomorphism of Voevodsky:
\begin{equation} \label{eq:motivic&Chow}
H^{2n}_\M\big(X,\ZZ(n)\big)=CH^n(X).
\end{equation}
\end{ex}

\num \label{funct_Chow}
In the sequel, we will 
need the following functoriality of the Chow group
of cycles with coefficients in a cycle module $\phi$~:
\begin{itemize}
\item $A^*(.;\phi)$ is contravariant for flat morphisms (\cite[(3.5)]{Ros}).
\item $A^*(.;\phi)$ is covariant for proper morphisms (\cite[(3.4)]{Ros}).
\item For any smooth scheme $X$, $A^*(X;\phi)$ is a graded module over 
 $CH^*(X)$ (\cite[5.7 and 5.12]{Deg4bis}).
\item $A^*(.;\phi)$ is contravariant for morphisms between smooth schemes
 (\cite[\textsection 12]{Ros}).
\end{itemize}
Note that any morphism of cycle modules gives a transformation
 on the corresponding Chow group with coefficients
  which is compatible with the functorialities listed above.
Moreover, identifying $A^p(.;K_*^M)_p$ with $CH^p(.)$,
 following the preceding example,
  the structures above correspond to the usual structures on the Chow 
   group.
Finally, let us recall that the maps appearing
 in the first three points above are defined
 at the level of the complexes $C^*(.;\phi)$ (introduced in \ref{num:cycle_complex}).

%
%

In \cite{BO}, the authors expressed the $E_2$-term of
the coniveau spectral sequence as the Zariski cohomology of a
well defined sheaf. We get the same result in the motivic setting.
Let $\cH_*^{q,n}$ be the presheaf of graded abelian groups
 on the category of smooth schemes such that
\begin{equation} \label{eq:unramif_cohomology}
\Gamma\big(X;\cH_*^{q,n})\big):=A^0\big(X;\hat H_*^{q,n}\big).
\end{equation}
Classically, this group is called the {\it $n$-th twisted unramified cohomology} 
of $X$ with coefficients in $H$.
\begin{prop} \label{prop:coniv_spectral3}
Consider the notations above.
\begin{enumerate}
\item The presheaf $\cH_*^{q,n}$ 
 has a canonical structure of a homotopy invariant sheaf with transfers.
\item There is a isomorphism of abelian groups:
$$
A^p(X;\hat H^{q,n}_*)=H^p_\mathrm{Zar}(X;\cH_*^{q,n})
$$
which is natural with respect to contravariant functoriality
 in the smooth scheme $X$.
\end{enumerate}
\end{prop}
\begin{proof}
The first assertion follows from \cite[(8.6)]{Ros}
 and \cite[6.9]{Deg4bis}
 while the second one follows from \cite[(2.6)]{Ros}.
\end{proof}

\begin{num}
Using the notations of the previous proposition,
 we have obtained the following form
 of the $\ZZ$-graded spectral sequence \eqref{coniv_spectral_graded}:
\begin{equation}
\label{coniv_spectral3}
E_2^{p,q}(X,n)=H^p_\mathrm{Zar}\big(X;\cH_*^{q,n}\big)
 \Rightarrow H^{p+q,n}_*(X).
\end{equation}
This is the analog of the Corollary 6.3 of \cite{BO}
 except for the definition of the sheaf $\cH_*^{q,n}$.
 However, using the argument of \emph{loc. cit.},
  we can recover the form considered by Bloch and Ogus
  for the sheaf $\cH^{q,n}_*$.
Indeed, the spectral sequence \eqref{coniv_spectral_graded}
 is natural with respect to open immersions.
 Thus, it can be sheafified for the Zariski topology
  and we obtain a spectral sequence of Zariski sheaves
  with coefficients in $\A^\ZZ$, converging
  to the Zariski sheaf $\tilde{\cH}_*^{q,n}$
  associated with the presheaf:
$$
X \mapsto H^{q,n}_*(X).
$$
According to the preceding computation of the $E_2$-term,
 we obtain that the sheafified spectral sequence is concentrated
 in the line $p=0$ from $E_2$ on.
 This gives an isomorphism $\tilde{\cH}_*^{q,n} \simeq \cH_*^{q,n}$
 as required. Let us state this:
\end{num}
\begin{prop} \label{prop:comput_unramified_sheaf}
The sheaf $\cH_*^{q,n}$ defined by formula \eqref{eq:unramif_cohomology}
 is equal to the Zariski sheaf on $\sm$ associated with $H^{q,n}_*$.
\end{prop}

\begin{ex} \label{ex:MWC}
In \cite{CD3}, Cisinski and the author have introduced axioms
 on a presheaf of differential graded algebras $E$ over smooth affine schemes
 which guarantee the existence of a realization functor
$$
H_E:\dmgm^{op} \rightarrow K-vs
$$
such that for any smooth affine scheme $X$,
 $H_E(M(X)[-i])=H^i\big(E(X)\big)$.
 We call $E$ a \emph{mixed Weil theory}. 
A distinctive feature of the resulting cohomology
 is that it is periodic with respect to the twist ;
 in other words, there exists an isomorphism:
$$
\epsilon_M:H_E(M(1)) \rightarrow H_E(M)
$$
which can be chosen to be natural in $M$.

Let us summarize the properties obtained previously for this particular kind
 of realization:
\begin{itemize}
\item There exists canonical cycle modules $\hat H^{q,n}_{E,*}$
 such that for any function field $L/k$,
$$
\hat H^{q,n}_{E,r}(L):=\ilim{A/k} H^{q+r,n+r}\big(\spec A\big)
$$
where the limit runs over the models of $L/k$ -- see Par. \ref{num:model}.
In fact, this family of cycle modules is equivalent
 to only one of them according to the following isomorphism:
$$
\hat H^{q,n}_{E,*}=\hat H^{0,n-q}_{E,*-q} \simeq \hat H^{0,0}_{E,*-q}
$$
where the equality follows from Remark \ref{rem:relation_generic}
 and the isomorphism is induced by $\epsilon$.
\item Let $\cH^{q,n}_{E,*}$ be the Zariski sheaf on $\sm$
 associated with the presheaf $H^{q,n}_{E,*}$.
 Then $\cH^{q,n}_*$ has a canonical structure of a homotopy invariant sheaf
  with transfers. Moreover, $\cH^{q,n}_*(X)$ is constant on the birational
  class of a smooth proper scheme $X$.

There exists a spectral sequence associated with the cohomology $H_E$,
 converging to its coniveau filtration (recall formula \eqref{eq:coniveau_filtration}),
 of the following form:
$$
E_2^{p,q}(X,n)=A^p(X,\hat H^{q,n}_{E,*})=H^p_\mathrm{Zar}\big(X,\cH^{q,n}_{E,*}\big)
 \Rightarrow H^{p+q,n}_{E,*}(X).
$$
\end{itemize} 
\end{ex}

\begin{rem} \label{rigid_realization}
Assume $k$ is a perfect field of characteristic $p>0$,
 let $W$ be the associated Witt ring and denote by $K$ the fraction field of $W$.
 According to \cite[sec. 3.2]{CD2}, there exists a mixed Weil theory $E$
 such that:
\begin{itemize}
\item for any smooth affine scheme $X$, $H^n_E(X)$ is the Monsky-Washnitzer
 cohomology (of a lift of $X$ over $W$) -- see \cite{MW}.
\item for any smooth proper scheme $X$, $H^n_E(X)$ is the crystalline
 cohomology of $X/W$ tensored with $K$.
\end{itemize}
Then, the preceding example applied to this mixed Weil theory,
 together with forthcoming Corollary \ref{Bloch-Ogus_algebraic} gives the
 results stated in the end of the introduction.
\end{rem}

\begin{num} \textit{Regulators}.-- Consider again the situation and notations
 of the previous example. The algebra structure on $E$ induces an algebra structure
 on $H^{**}_E$. The unit of this structure $1 \in H^{0,0}_E(\spec k)$ corresponds
 to an element $\rho \in H_E(\ZZ)$ which induces a natural transformation
$$
\rho:H_\M \rightarrow H_E
$$
because motivic cohomology corresponds to the functor $\Hom_{\dmgm}(?,\ZZ)$.
 This is the \emph{regulator map} -- or rather, its extension to the full triangulated
 category of mixed motives $\dmgm$.
 According to the preceding construction, this map induces a natural transformation
 of cycle modules:
$$
\hat H_{\M,*}^{q,n} \rightarrow \hat H_{E,*}^{q,n}
$$
corresponding to what should be called \emph{higher symbols}.
Indeed, in the case $q=n$, it gives usual symbols for any function field $L$:
$$
K_n^M(L) \rightarrow \hat H^{n,n}_E(L).
$$
This higher symbol map is of course compatible with all the structures
 of a cycle module: corestriction, restriction, residues.
\end{num}

\subsection{Algebraic equivalence}

\begin{num} \label{num:axioms_realization}
In this section,
 we assume $\A$ is the category of $K$-vector spaces
 for a given field $K$.

Consider the cycle modules $\hat H^{q,n}_*$ associated with $H$ in Definition
 \ref{generic_transform4real}, together with their $\ZZ$-grading
 defined by formula \eqref{Rost_grading_unramified_cycle_module}.
 We introduce the following properties on the realization functor $H$:
\begin{description}
\item[(Vanishing)] For any function field $E$
 and any couple of negative integers $(q,n)$, $\hat H_0^{q,n}(E)=0$.
\item[(Rigidity)] The covariant functor $\hat H^{0,0}_0$ on function fields over $k$
 is the constant functor with value $K$.
\end{description}

Let us assume $H$ satisfies (Rigidity). 
 Then the unit element of the field $K$ determines
 an element $\rho$ of $H(\ZZ)$ through the identification
$$
K=\hat H^{0,0}_0(k)=H(\ZZ).
$$
We deduce from $\rho$, as in example \ref{ex:MWC},
 a canonical natural transformation of contravariant functors on $\dmgm$:
\begin{equation} \label{eq:higher_cyc_class}
\rho:\Hom_{\dmgm}(.,\ZZ) \rightarrow H.
\end{equation}
Of course, the source functor is nothing else
 than the realization functor which corresponds to motivic cohomology $H_\M$.
 In particular, taking of Voevodsky's isomorphism recalled in
 \eqref{eq:motivic&Chow}, we get a canonical cycle class:
$$
\rho^n_X:CH^n(X)_K \rightarrow H^{2n,n}(X).
$$
Let us denote by $Z^n(X,K)$ the group of $n$-codimensional $K$-cycles in $X$
 (simply called \emph{cycles} in what follows)
 and by $\K^n_{rat}(X)$ (resp. $\K^n_{alg}(X)$)
 its subgroup formed by cycles rationally (resp. algebraically)
 equivalent to $0$.
\end{num}
\begin{df}
Using the notations above,
 we define the group of cycles $H$-equivalent to $0$
 as:
$$
\K^n_H(X)=\{\alpha \in Z^n(X,K) \mid \rho^n_X(\alpha)=0\}.
$$
\end{df}

\begin{rem} \label{rem:symbol}
The map \eqref{eq:higher_cyc_class}
induces a morphism of cycle modules
$K_{*+a}^M \rightarrow \hat H^{a,a}_*$ which corresponds
to cohomological symbols
$K_a^M(E) \rightarrow \hat H^{a,a}(E)$
compatible with corestriction, restriction, residues and the
action of $K_*^M(E)$.
\end{rem}

\begin{num}
We analyze the coniveau spectral sequence \eqref{coniv_spectral}
 under the assumptions (Vanishing) and (Rigidity). 
 The $E_1$-term is described by the following picture:
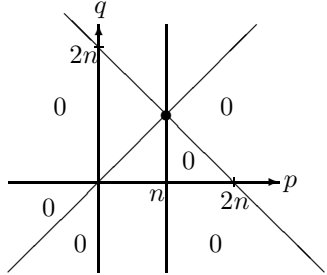
\begin{center}
\unitlength=0.3cm
\begin{picture}(20,12)
\put(4,0){\vector(0,1){11}}
\put(0,4){\vector(1,0){12}}
\put(4,10){\line(1,-1){10}}
\put(4,10){\line(-1,1){1.5}}
\put(7,0){\line(0,1){11}}
\put(0,0){\line(1,1){11}}
\put(6.7,6.7){$\bullet$}
\put(12.2,3.8){$p$}
\put(3.8,11.5){$q$}
\put(7,3.9){\line(0,1){0.2}}
\put(10,3.8){\line(0,1){0.4}}
\put(3.8,10){\line(1,0){0.4}}
\put(6.25,3.2){$n$}
\put(9.4,2.8){$2n$}
\put(2.7,9.3){$2n$}
\put(2.9,0.9){0}
\put(1.5,2.5){0}
\put(2,7){0}
\put(8.9,0.9){0}
\put(7.7,4.6){0}
\put(9.4,7){0}
\end{picture}
\end{center}
Property (Rigidity)
implies that $E_1^{n,n}(X,n)=Z^n(X,K)$.
As only one differential goes to $E_r^{n,n}$,
we obtain a sequence of epimorphisms:
$$
Z^n(X,K)=E_1^{n,n}(X,n) \rightarrow E_2^{n,n}(X,n) \rightarrow E_3^{n,n}(X,n)
 \rightarrow \hdots
$$
which become isomorphisms as soon as $r>n$.
Thus, if we put 
$$
\K_{(r)}^n(X)=\mathrm{Ker}(E_1^{n,n}(X,n) \rightarrow E_{r+1}^{n,n}(X,n)),
$$
we obtain an increasing filtration on $Z^n(X,K)$:
\begin{equation} \label{eq:equiv_rel_H}
 \K_{(1)}^n(X) \subset \K_{(2)}^n(X)
 \subset \hdots \subset \K_{(n)}^n(X)\subset Z^n(X,K)
\end{equation}
such that $E_r^{n,n}(X,n)=Z^n(X,K)/\K_{(r-1)}^n(X)$.

Note also that $E_n^{n,n}=E_\infty^{n,n}$ is the first step
of the coniveau filtration on $H^{2n}(X,n)$ so that we get a monomorphism
$$
\epsilon:E_n^{n,n}(X,n) \rightarrow H^{2n,n}(X).
$$
Note these considerations can be applied to the functor
 $\Hom_{\dmgm}(.,K)$ corresponding to $K$-rational motivic cohomology.
 In this case, according to Example \ref{ex:motivic&Milnor},
 the $E_r^{n,n}=CH^n(X)_K=H^{2n}_{\mathcal M}(X;K(n))$.

Returning to the general case,
 the natural transformation $\rho$ induces
 a morphism of the coniveau spectral sequences.
 This induces the following commutative diagram:
\begin{equation} \label{eq:ultim_descr_coniv_ssp}
\begin{split}
\xymatrix@R=6pt{
& CH^n(X)_K\ar@{=}[r]\ar@{->>}^{\tilde \rho^n_X}[dd]
 & CH^n(X)_K\ar^-\sim[r]\ar@{->>}[dd]
  & H^{2n}_{\mathcal M}(X;K(n))\ar^{\rho^n_X}[dd] \\
Z^n(X,K)\ar@{->>}[ru]\ar@{->>}[rd] &&& \\
&E_2^{n,n}(X,n)\ar@{->>}[r] & E_n^{n,n}(X,n)\ar@{^(->}^\epsilon[r] & H^{2n,n}(X)
}
\end{split}
\end{equation}
\end{num}

The following proposition is a generalization of a result
of Bloch-Ogus (cf \cite[(7.4)]{BO}).
\begin{prop} \label{Bloch-Ogus}
Consider the preceding hypothesis and notations.
Then the following properties hold:
\begin{enumerate}
\item[(i)] For any scheme $X$ and any integer $n \in \NN$,
 $\K_{rat}^n(X) \subset \K_{(1)}^n(X)$.
\item[(ii)]  For any scheme $X$ and any integer $n \in \NN$,
 $\K_{(n)}^n(X)=\K_H^n(X)$.
\end{enumerate}
Moreover, the following conditions are equivalent~:
\begin{enumerate}
\item[(iii)] For any smooth proper scheme $X$, 
 $\K^1_H(X)=\K^1_{alg}(X)$.
\item[(iii')]
For any smooth proper scheme $X$ and any $n \in \NN$,
 $\K^n_{(1)}(X)=\K_{alg}^n(X)$.
\end{enumerate}
\end{prop}
Note that under the equivalent conditions (iii) and (iii'),
 the morphism $\tilde \rho^n_X$ induces,
  according to \eqref{coniv_spectral3}, an isomorphism:
\begin{equation} \label{eq:cycles_mod_alg&unramified}
A^n(X)_K \xrightarrow{\ \sim\ } H^n_\mathrm{Zar}\big(X;\cH^{n,n}\big)
\end{equation}
 where $\cH^{n,n}$ is the Zariski sheaf associated
 with $H^{n,n}$ -- apply propositions \ref{prop:coniv_spectral3}
  and \ref{prop:comput_unramified_sheaf}.
\begin{proof}
Properties (i) and (ii) are immediate consequences
 of \eqref{eq:ultim_descr_coniv_ssp}.

Note that, for $n=0$, condition (iii') always holds.
Obviously (iii) implies (iii') according to assertion (ii).
Thus it remains to prove that (iii) implies (iii') when $n>1$.

Fix an integer $n>1$.
We first prove the inclusion $\K^n_{alg}(X) \subset \K^n_{(1)}(X)$.
Consider cycles $\alpha, \beta \in Z^n(X,K)$ such that $\alpha$ is
algebraically equivalent to $\beta$.
This means
there exists a smooth proper connected curve $C$,
 points $x_0,x_1 \in C(k)$,
  and a cycle $\gamma$ in $Z^n(X \times C,K)$ such that
  $f_*(g^*(x_0).\gamma)=\alpha$, $f_*(g^*(x_1).\gamma)=\beta$
where $f:X \times C \rightarrow X$ and $g:X \times C \rightarrow X$
are the canonical projections. Using the functoriality described in
paragraph \ref{funct_Chow} applied to the morphism of cycle modules 
 $K_{*}^M \rightarrow \hat H_*^{0,0}$ (Remark \ref{rem:symbol}),
 we get a commutative diagram
$$
\xymatrix@C=16pt@R=14pt{
A^1(C;K_*^M)_K\ar^-{q^*}[r]\ar_{(1)}[d]
 & A^1(C \times X;K_*^M)_K\ar^-{.\gamma}[r]\ar[d]
  & A^{p+1}(C \times X;K_*^M)_K\ar^-{f_*}[r]\ar[d]
   & A^n(X;K_*^M)_K\ar^{(2)}[d] \\
A^1(C;\hat H_*^{0,0})\ar^-{q^*}[r]
 & A^1(C \times X;\hat H_*^{0,0})\ar^-{.\gamma}[r]
  & A^{p+1}(C \times X;\hat H_*^{0,0})\ar^-{f_*}[r]
   & A^n(X;\hat H_*^{0,0}) \\
}
$$
Recall the identifications:
$$
A^n(X;K_*^M)_n=CH^n(X), \quad
 A^n(X;\hat H_*^{0,0})_n=A^n(X;\hat H_*^{n,n})_0=E_2^{n,n}(X,n).
$$
According to these ones, the first (resp. $n$-th) graded piece of the map (1)
(resp. (2)) can be identified with the morphism 
 $\tilde \rho^1_X$ (resp. $\tilde \rho^n_X$).
In particular, we are reduced to prove that $x_0-x_1$ belongs to $\K^1_{(1)}(C)$.
This finally follows from (iii).

We prove conversely that $\K^n_{(1)}(X) \subset \K^n_{alg}(X)$.
Recall $A^n(X;\hat H_*^{n,n})_0$
is the cokernel of the differential \eqref{eq:diff_modl}
$$
C^{n-1}(X;\hat H_*^{n,n})_0
 \xrightarrow{d^{n-1}} C^n(X;\hat H_*^{n,n})_0=Z^n(X,K).
$$
We have to prove that the image of
this map consists of the cycles algebraically equivalent to zero.
Consider a point $y \in X^{(p-1)}$ with residue field $E$
 and an element $\rho \in \bar H^{1,1}(E)$. 
Let $i:Y \rightarrow X$ be the immersion  of the reduced closure of $y$ in $X$.
Using De Jong's theorem,
 we can consider an alteration $Y' \xrightarrow f Y$
 such that $Y'$ is smooth over $k$.
Let $\varphi:E \rightarrow L$ be the extension of function 
fields associated with $f$.
According to the basic functoriality of cycle modules
 recalled in Paragraph \ref{funct_Chow},
 we obtain a commutative diagram
$$
\xymatrix@R=8pt@C=16pt{
\bar H^{1,1}(L)\ar@{=}[r]\ar_{N_{L/E}}[d]
 & C^0(Y';\hat H_*^{1,1})_0\ar^-{d^1_{Y'}}[r]\ar[d]
 & C^1(Y';\hat H_*^{1,1})_0\ar@{=}[r]\ar[d]
 & Z^1(Y')\ar^{f_*}[d] \\
\bar H^{1,1}(E)\ar@{=}[r]
 & C^0(Y;\hat H_*^{1,1})_0\ar^-{d^1_{Y}}[r]\ar@{^(->}[d]
 & C^1(Y;\hat H_*^{1,1})_0\ar@{=}[r]\ar[d]
 & Z^1(Y)\ar^{i_*}[d] \\
& C^{n-1}(X;\hat H_*^{n,n})_0\ar^-{d^{n-1}_{X}}[r]
 & C^n(X;\hat H_*^{n,n})_0\ar@{=}[r]
 & Z^n(X) \\
}
$$
where $f_*$ and $i_*$ are the usual proper pushouts on cycles.
Recall from \cite[(R2d)]{Ros} that
 $N_{L/E} \circ \varphi_*=[L:E].Id$ for the cycle module $\hat H_*^{1,1}$.
Thus, $N_{L/E}$ is surjective.
As algebraically equivalent cycles are stable by direct images
of cycles, we are reduced to the case of the scheme $Y'$, 
in codimension $1$, already obtained above.
\end{proof}

\begin{rem}
In the preceding proof,
 if we can replace the alteration $f$ by a 
 (proper birational) resolution of singularities, 
then the theorem is true with integral coefficients
 -- indeed, the extension $L/E$ which shows up in the
 end of the proof is trivial when $f$ is birational.
This holds in characteristic $0$ by Hironaka's resolution
 of singularities but also in characteristic $p>0$
 if $X$ is a curve, a surface (cf \cite{Lipman}) 
 or a 3-fold (cf \cite{CP}).
\end{rem}

\begin{num}
We consider the assumptions and notations of Example \ref{ex:MWC}.
An important property of a mixed Weil theory is the fact the graded functor:
$$
H^*_E:\dmgm^{op} \rightarrow (K-vs)^\ZZ, M \mapsto \big(H^n_E(M),\ n \in \ZZ \big)
$$
is monoidal where the target category is the monoidal category
 of $\ZZ$-graded vector spaces.

Recall from \cite[Prop. 2.18]{Deg6} that for any smooth projective scheme $X$
 of pure dimension $d$,
 there exists a strong duality pairing 
 $\eta:M(X) \otimes M(X)(-d)[-2d] \rightarrow \ZZ$.
 Applying to this pairing the monoidal functor $H^*_E$,
  we get the usual Poincar\'e duality pairing:
$$
H^{q,n}_E(X) \otimes H^{2d-q,d-n}_E(X) \rightarrow K.
$$
As in the above, the regulator map induces a morphism
 of the unramified sheaves $\cH_M^{q,n} \rightarrow \cH^{q,n}_E$
 which induces an "unramified" regulator:
$$
\tilde \rho_X:CH^n(X) \simeq H^n_\mathrm{Zar}(X,\cH^{n,n}_\M)
 \rightarrow H^n_\mathrm{Zar}(X,\cH^{n,n}_E)
$$
As a corollary of the preceding proposition,
 we get the following result:
\end{num}
\begin{cor} \label{Bloch-Ogus_algebraic}
Consider the notations above.
For any pair of integers $(q,n)$,
 let $\cH^{q,n}_E$ be the Zariski sheaf on the category of smooth schemes
 associated with $H^{q,n}_E$.

Assume the realization functor $H_E$
 satisfies property (Vanishing) (\textsection \ref{num:axioms_realization}).
Then, the following conditions are equivalent~:
\begin{enumerate}
\item[(i)] The realization functor $H_E$
 satisfies property (Rigidity) (\textsection \ref{num:axioms_realization}).
\item[(ii)] For any integer $n \in \NN$
 and any projective smooth scheme $X$,
 the unramified regulator map $\tilde \rho_X$ considered above
 induces an isomorphism
$$
A^n(X)_K \rightarrow H^n_{\mathrm{Zar}}(X;\cH^{n,n}_E).
$$
\end{enumerate}
\end{cor}
\begin{proof}
Remark the assumption implies that for any smooth scheme $X$ and
 any $i<0$, $H^i(X,\E)=0$ -- apply the coniveau spectral sequence for $X$.

$(i) \Rightarrow (ii)$~: According to our hypothesis,
 we can apply Proposition \ref{Bloch-Ogus} to the realization functor $H_E$.
Indeed, we have assumed (Vanishing) and (Rigidity).
Moreover, Property (Rigidity) and the Poincar\'e duality
pairing implies that for any smooth projective connected curve
$p:C \rightarrow \spec k$, the morphism 
$p_*:H^2(C,\E)(1) \rightarrow H^0(C,\E)=K$ is an isomorphism.
Following classical arguments, this together with the multiplicativity
of the cycle class map implies
that homological equivalence for $\E$
is between rational and numerical equivalence.
From Matsusaka's theorem (cf \cite{Mat}), 
these two equivalences coincide for divisors. 
This implies assumption (iii) of Proposition \ref{Bloch-Ogus},
 and we can conclude from the isomorphism \eqref{eq:cycles_mod_alg&unramified}.

$(ii) \Rightarrow (i)$~: 
For a $d$-dimensional smooth projective connected scheme $X$, 
we deduce from the coniveau spectral sequence and Poincar\'e
duality  
that $E_2^{d,d}(X,d)=H^{2d}(X,\E)(d)=H^0(X,\E)$.
Thus property (ii) implies $H^0(X,\E)=K$. If $L$ is the function field of
$X$, we deduce that $\bar H^0(L,\E)=K$.
Considering any function field $E$, we easily construct 
an integral projective scheme $X$ over $k$ with function
field $E$. Applying De Jong's theorem, we find an alteration 
$\tilde X \rightarrow X$ such that $\tilde X$ is projective 
smooth and the function field $L$ of $\tilde X$ is a finite 
extension of $E$ and the result now follows from the fact
$N_{L/E}:\bar H^0(L) \rightarrow \bar H^0(E)$ is a split
epimorphism.
\end{proof}

\begin{rem}
\begin{enumerate}
\item Condition (i) in the previous corollary is only reasonable
when the base field $k$ is separably closed (or after an extension
to the separable closure of $k$).
\item When $k$ is the field of complex numbers
 and $H$ is algebraic De Rham cohomology,
 the filtration on cycles \eqref{eq:equiv_rel_H}
 is usually called the \emph{Bloch-Ogus filtration} -- see \cite{Fri}.
 It can be compared with other filtrations
  (see \cite{Nori}, \cite{Fri}).
 It is an interesting question whether a similar comparison
 to that of \cite[rem. 5.4]{Nori} can be obtained in the case of rigid cohomology.
\end{enumerate}
\end{rem}

\bibliographystyle{alpha}
\bibliography{gysin}

\end{document}